\documentclass{amsart}

\usepackage{amssymb}

\newtheorem{theorem}{Theorem}[section]
\newtheorem{lemma}[theorem]{Lemma}
\newtheorem{proposition}[theorem]{Proposition}
\newtheorem{corollary}[theorem]{Corollary}
\theoremstyle{definition}

\theoremstyle{remark}
\newtheorem{remark}[theorem]{Remark}
\numberwithin{equation}{section}

\newcommand{\R}{{\mathbb R}}
\newcommand{\Z}{{\mathbb Z}}

\newcommand{\e}{\epsilon}

\title[Parabolic Problem]{A parabolic problem with a fractional time derivative}
\author{Mark Allen}
\address{Department of Mathematics, The University of Texas at Austin,
Austin, TX 78712, USA}
\email{mallen@math.utexas.edu}

\author{Luis Caffarelli}
\address{Department of Mathematics, The University of Texas at Austin,
Austin, TX 78712, USA}
\email{caffarel@math.utexas.edu}

\author{Alexis Vasseur}
\address{Department of Mathematics, The University of Texas at Austin,
Austin, TX 78712, USA}
\email{vasseur@math.utexas.edu}

\begin{document}

\begin{abstract}
 We study regularity for a parabolic problem with fractional diffusion in space and a fractional time derivative. 
 Our main result is a De Giorgi-Nash-Moser H\"older regularity theorem for solutions in a divergence form equation.
 We also prove results regarding existence, uniqueness, and higher regularity in time. 
\end{abstract}

\maketitle

\section{Introduction}
 In this paper we study the regularity of nonlocal evolution equations with ``measurable'' coefficients in space and
 a fractional time derivative. Anomolous diffusion equations are of great interest in physics. Fractional 
 diffusion operators can arise in the context of levy flights, see for instance the fractional kinetic equations 
 in \cite{z02}. 
 Fractional kinetic equations can also be derived from the context of random walks. A fractional diffusion operator
 corresponds to a diverging jump length variance in the random walk, and a fractional time derivative arises when
 the characteristic waiting time diverges, see \cite{mk00}. A fractional time derivative models situations in
 which there is ``memory''. 
 
 The problem we consider uses the Caputo fractional time derivative. The caputo derivative has been used recently
 to model fractional diffusion in plasma turbulence, see \cite{d04} and \cite{d05}. In the physical model
 in \cite{d05} the Caputo derivative accounts for the trapping effect of turbulent eddies. 
 
 Another advantage of using the Caputo 
 derivative in modeling physical problems is that the Caputo derivative of constant functions is zero. 
 Thus, time-indepent solutions are also solutions of the time-dependent problem. We note this is not the
 case for the Riemann-Liouville time derivative. 
 
 The specific equation we study is 
  \begin{equation}  \label{e:first}
   _{a}^cD_t^{\alpha} w(t,x) = \int[w(t,y)-w(t,x)]K(t,x,y)dy+f(t,x).
  \end{equation}
 We impose a symmetry condition on the kernel $K$: 
  \[
   K(t,x,y)=K(t,y,x) \quad \text{ for any } \quad x\neq y.
  \]
 We also assume an ellipticity condition. We assume there exists $0< \sigma < 2$ and $0< \Lambda$, such that
  \begin{equation}  \label{e:elliptic}
   \chi_{\{|x-y|\leq 3\}} \frac{1}{\Lambda}|x-y|^{-(N+\sigma)} 
    \leq K(t,x,y) 
     \leq  \Lambda|x-y|^{-(N+\sigma)},
  \end{equation}
 where $N$ is the spatial dimension. 
 The symmetry of the kernel $K$ allows us to consider the divergence form  
  \begin{equation}  \label{e:e}
   \begin{aligned}
    &\int_a^T \int_{\R^N} \int_{\R^N}K(t,x,y)[w(t,x)-w(t,y)][\phi(t,x)-\phi(t,y)] \ dx \ dy \ dt \\
    &\quad +\int_a^T \int_{\R^N} \  _{a}^cD_t^{\alpha} w(t,x) \phi(t,x) =\int_a^T\int_{\R^n}f(t,x)\phi(t,x).
   \end{aligned}
  \end{equation}
 for every $\phi \in C_0^1(\R^n) \times C^1(0,T)$. We also assume $f \in L^{\infty}$. This assumption is most likely
 not optimal; however, we assume $f \in L^{\infty}$ to make clear the method. Also, our interest in 
 $f \in L^{\infty}$ is to be able to prove higher regularity in time in Section \ref{s:highertime}. 
 Furthermore, the right hand side in the physical model in \cite{d05} is in $L^{\infty}$.

 A very similar problem with zero right hand side was recently studied by the second and third author in \cite{ccv11} 
 with the standard local time derivative. In that paper
 the original method of De Giorgi was used to prove boundedness of solutions and local H\"older regularity. 
 In this paper we use a similar approach to prove apriori local H\"older estimates of solutions to 
 \eqref{e:e}. We follow the De Giorgi method as in \cite{ccv11} but take into account the fractional time 
 derivative. In the case of the second Lemma in Section \ref{s:lemma2}, we utilize the fractional derivative
 to give a simpler proof. However, by utilizing the fractional nature of the derivative our estimates do 
 not remain uniform as $\alpha \to 1$. 
 
 We note that a similar problem to \eqref{e:e} was recently studied by Zacher in \cite{z13}. That problem 
 had a local diffusion equation in divergence form with a Riemann-Liouville fractional time derivative and zero
 right hand side. 
 \subsection{Overview of the Main results}
  The main result of this paper is a De Giorgi-Nash-Moser type H\"older regularity theorem 
   (see Theorem \ref{t:holder}) for a certain class of weak solutions defined in \eqref{e:w}. These solutions are
   weak solutions in space and very weak in time. If $w$ is a solution to \eqref{e:w} it is not known a priori whether
   cut-offs of $w$ are valid test functions. To utilize energy estimate techniques we therefore consider a sequence of
   discretized-in-time approximating solutions which will be strong in time. These approximating solutions will also 
   prove existence of solutions (see the appendix.) We utilize the ideas of De Giorgi's first and second lemmas applied
   to the approximating solutions. In the limit we obtain the desired H\"older regularity.
    
   There are several consequences of the H\"older regularity. If the kernel $K$ and right hand side $f$ are 
   sufficiently regular in time, then solutions to \eqref{e:w} will have higher regularity in time (see 
   Theorem \ref{t:smooth}). Under additional regularity conditions - in space -  at the intial time $t=a$, solutions
   to \eqref{e:w} are continuous up to the initial time (Lemma \ref{l:initime}), and hence we are able to show
   that such solutions will also be strong in time and solutions to \eqref{e:e} (see Corollary \ref{c:strong}). 
   For such strong conditions we obtain the usual corollary of uniqueness (Corollary \ref{c:unique}).   
   
   A further application of Theorem \ref{t:holder} is in regards to the equation
    \[
     -\int_{\R^n} F'(\theta(y)-\theta(x))K(y-x) dy = 0
    \]
   which arises as the Euler-Lagrange equation of the variational integral
    \[
     V(\theta) = \int_{\R^n} \int_{\R^n} F(\theta(y) -\theta(x))K(y-x) \ dy \ dx.
    \]
   $F:\R \to [0,\infty)$ is assumed to be even, convex, $C^2(\R)$ and satisfy $F(0)=0$ and 
   $\Lambda^{-1/2} \leq F''(x) \leq \Lambda^{1/2}$. $K(x)=K(-x)$ and satisfies
    \[
     \chi_{\{ |x|\leq 3 \}} \frac{\Lambda^{-1/2}}{|x|^{n+\sigma}} \leq K(x) \leq \frac{\Lambda^{1/2}}{|x|^{n+\sigma}}.
    \]
   Our results apply to the generalized time-dependent solution with Lipschitz right hand side
    \begin{equation}  \label{e:19}
     _a D_t^{\alpha} - \int_{\R^n} \phi'(\theta(t,y)-\theta(t,x))K(y-x) dy =f(t,x).  
    \end{equation}
   By a change of variables and differentiating in space (see \cite{ccv11} for details)
   we obtain that $D_e \theta$ is a solution to \eqref{e:e}
   and hence H\"older continuous. Therefore, we obtain that $\nabla \theta$ is H\"older continuous.

 \subsection{Outline} 
  The outline of this paper will be as follows
  
   - In Section \ref{s:caputo} we state some properties of the Caputo derivative. 
   
   - In Section \ref{s:discretize} we give a discretized version of the Caputo derivative. The discretized
   version will be useful to both prove existence and to obtain our regularity results. Since our solutions
   are ``very weak'' in time the solution itself is not a valid test function. This is overcome by approximating
   via discretized solutions which are ``strong'' in time. 
   
   - In Section \ref{s:lemma1} we prove the first De Giorgi Lemma or ``$L^2, L^{\infty}$'' estimate. 
   
   - In Section \ref{s:lemma2} we prove the second De Giorgi Lemma. We utilize the nonlocal nature of the
   time derivative to simplify the proof. 
   
   - In Section \ref{s:Holder} we utilize the first and second De Giorgi Lemmas to prove the decrease in 
   oscillation and obtain H\"older regularity. 
   
   - In Section \ref{s:highertime} we prove higher regularity in time for solutions with appropriate conditions
   for the kernel $K$ and right hand side $f$. 
   
   - In the appendix we provide the details and prove existence of solutions via the discretized approximations.

 \subsection{Notation}
  We list here the notation that will be used consistently throughout the paper. The following letters are 
  fixed throughout the paper and always refer to:
   \begin{itemize}

   \item  $\alpha$ - the order of the Caputo derivative.
   
   \item  $\sigma$ - the order of the spatial fractional operator associated to the kernel $K(t,x,y)$. We use $\sigma$
    for the order because $s$ will always be a variable for time. 
   
   \item $a$ -  the initial time for which our equation is defined.
   
   \item $_a\partial_t^{\alpha}$ - the rescaled Caputo derivative as defined in Section \ref{s:caputo}.
   
   \item $\e$ - will always refer to the time length of the discrete approximations 
   as defined in Section \ref{s:discretize}
   
   \item $n$ - will always refer to the space dimension. 
           
 \end{itemize}

\section{Caputo Derivative} \label{s:caputo}
 In this section we state various properties of the Caputo derivative that will be useful.
 The Caputo derivative for $0<\alpha<1$ is defined by 
  \[
   _{a}^cD_{t}^{\alpha} u(t) := \frac{1}{\Gamma(1-\alpha)} \int_{a}^{t} \frac{u'(s)}{(t-s)^{\alpha}} \ ds
  \]
 For the remainder of the paper we will drop the superscript $c$ and understand that throughout
 the paper the fractional derivative is the Caputo derivative. 
 By using integration by parts we have the alternative formula
  \begin{equation} \label{e:fracalt}
    \Gamma(1-\alpha) \ _{a}D_{t}^{\alpha} u(t) = 
   \frac{u(t)-u(a)}{(t-a)^{\alpha}} + \alpha \int_{a}^t \frac{u(t)-u(s)}{(t-s)^{\alpha + 1}} \ ds
  \end{equation}
 For notational simplicity we use the rescaled Caputo derivative 
 $\partial_t^{\alpha} := \Gamma(1-\alpha) D_t^{\alpha}$ to avoid writing $\Gamma(1-\alpha)^{-1}$ 
 in all of our formulas. Also, for the remainder of the 
 paper we will drop the subscript $a$ on $_{a}\partial_t^\alpha$ when 
 the initial point $a$ is understood.  
  
 
 For a function $g(t)$ defined on $[a,t]$ and working with $\partial_t^{\alpha}$, 
 there are two advantageous ways of defining $g(t)$ for $t<a$. The first way is to define $g(t)\equiv 0$ 
 for $t<a$. The second way is to define $g(t)\equiv g(a)$ for $t<a$. When using the latter definition we
 note that
  \[
   _a \partial_t^{\alpha} g(t) = \  _{-\infty}\partial_t^{\alpha} g(t)
     = \alpha \int_{-\infty}^t \frac{g(t)-g(s)}{(t-s)^{1+\alpha}}.
  \]
 This looks very similar to $(-\Delta)^{\alpha}$ except the integration only occurs for $s<t$. In 
 this manner the Caputo derivative retains directional derivative behavior while at the same time sharing
 certain properties with $(-\Delta)^{\alpha}$. This is perhaps best illustrated by the following
 integration by parts formula for the Caputo derivative:  
 
 \begin{proposition}  \label{p:changev}
  Let $g,h\in C^1(a,T)$. Then 
   \[
    \begin{aligned}
    \int_{a}^T g\partial_t^{\alpha}h + h\partial_t^{\alpha} g &= 
     \int_{a}^T {g(t)h(t) \left[\frac{1}{(T-t)^{\alpha}} + \frac{1}{(t-a)^{\alpha}} \right]} \ dt \\
     & \quad + \alpha \int_{a}^T \int_{a}^t \frac{(g(t)-g(s))(h(t)-h(s))}{(t-s)^{1+\alpha}} \ ds \ dt \\
     & \quad -   \int_{a}^T \frac{g(t)h(a)+h(t)g(a)}{(t-a)^{\alpha}} \ dt.
    \end{aligned}
   \]
 \end{proposition}
  We remark that the integration by parts formula above seems to be a combination of 
   \[
      \int_{\R} g(t)(-\Delta)^{\alpha} h(t) = \int_{\R} (-\Delta)^{\alpha/2} g(t) (-\Delta)^{\alpha/2} h(t) 
   \]
  and
   \[   
       \int_{a}^T g(t) \partial_t h(t) + h(t)\partial_t g(t) =g(T)h(t)-g(a)h(a). 
   \] 
 We omit the proof of the above Proposition since we do not use the Proposition 
 and since the proof is very similar to the proof of Lemma \ref{l:decompose}.

  With an integration by parts formula in hand we give the exact formulation of solutions which we study. 
  We assume $\phi \in C_0^2(\R^n) \times C^{1}([a,T])$.
  \begin{equation}   \label{e:w}
   \begin{aligned}
    &\int_{\R^n}\int_{a}^T {w(t,x)\phi(t,x) \left[\frac{1}{(T-t)^{\alpha}} 
         + \frac{1}{(t-a)^{\alpha}} \right]} \ dt \ dx \\
    & \  + \alpha \int_{\R^n} \int_{a}^T \int_{a}^t 
          \frac{(w(t,x)-w(s,x))(\phi(t,x)-\phi(s,x))}{(t-s)^{1+\alpha}} \ ds \ dt  \ dx \\ 
     &\ -  \int_{\R^n}\int_{a}^T \frac{\phi(t,x)w(a,x)+\phi(0,x)w(t,x)}{(t-a)^{\alpha}} \ dt \\
     &\   - \int_{\R^n}\int_a^T w(t,x) \partial_t^{\alpha} \phi(t,x) \ dt \ dx\\
     &\  +\int_a^T \int_{\R^n} \int_{\R^n} K(t,x,y)[w(t,x)-w(t,y)][\phi(t,x)-\phi(t,y)] \ dx \ dy \ dt \\
     & \qquad = \int_a^T \int_{\R^n} f(t,x) \phi(x,t). 
   \end{aligned}
  \end{equation}
 
  

 For a function $u(t)$ defined on $(a,T)$, defining $u(t)=0$ for $t<a$ is 
 useful when obtaining an energy estimate as follows
 \begin{lemma}  \label{l:ext}
  Let $u \in C([a,T])$. If we extend $u$ to all of $\R$ by having $u(t)=0$ for $t<a$ and 
  then reflecting evenly across $T$, we obtain
   \[
    \alpha \int_{\R} \int_{\R} \frac{|u(t)-u(s)|^2}{|t-s|^{1+\alpha}} \ ds \ dt \leq 
    8 \left( \alpha \int_{a}^T \int_{a}^t \frac{|u(t)-u(s)|^2}{|t-s|^{1+\alpha}} \ ds \ dt +
    \int_a^T \frac{u^2(t)}{(t-a)^{\alpha}} \right)
   \]
 \end{lemma}

 \begin{proof}
  We first note that since the integrand is symmetric in $s$ and $t$ we have 
   \[
    \int_{a}^T \int_{a}^T \frac{|u(t)-u(s)|^2}{|t-s|^{1+\alpha}} \ ds \ dt = 
     2\int_{a}^T \int_{a}^t \frac{|u(t)-u(s)|^2}{|t-s|^{1+\alpha}} \ ds \ dt.
   \]
  By the even reflection across the point $T$ we then have 
   \[
    \int_{a}^{2T-a} \int_{a}^{2T-a} \frac{|u(t)-u(s)|^2}{|t-s|^{1+\alpha}} \ ds \ dt \leq 
     8 \int_{a}^T \int_{a}^t \frac{|u(t)-u(s)|^2}{|t-s|^{1+\alpha}} \ ds \ dt.
   \]
  Since $u(t)=0$ if $|t|\leq|a|$ we only have to consider when $t \in (a,2T-a) , s \notin (a,2T-a)$ and vice-versa.
   \[
    \begin{aligned}
    \alpha \int_a^{2T-a} \int_{-\infty}^a \frac{|u(t)-u(s)|^2}{|t-s|^{1+\alpha}} \ ds \ dt &=
      \alpha \int_a^{2T-a} u^2(t) \int_{-\infty}^a |t-s|^{-1-\alpha} \ ds \ dt \\
      &= \int_a^{2T-a} \frac{u^2(t)}{(t-a)^{\alpha}} \ dt \\
      &\leq 2\int_a^{T} \frac{u^2(t)}{(t-a)^{\alpha}} \ dt. 
    \end{aligned}
   \] 
  The other three remaining pieces of integration are bounded exactly in the same manner.  
 \end{proof} 

 We will later need the following estimate
 \begin{lemma}  \label{l:fracbound}
  Let 
   \[
     h(t):= \max \{ |t|^{\nu} -1 , 0 \} \\
   \] 
    with $\nu < \alpha$. Then
   \[
    \partial_t^\alpha h \geq  -c_{\nu,\alpha}
   \]
  for $t \in \R$. Here, $c_{\nu,\alpha}$ is a constant depending only on $\alpha , \nu$. 
 \end{lemma}
 
 \begin{proof}
  By definition
   \[
    \partial_t^\alpha h =  \int_{a}^{t}\frac{-\nu|s|^{\nu-1}}{(t-s)^{\alpha}} \ ds
     \geq  \int_{- \infty}^{t}\frac{-\nu|s|^{\nu-1}}{(t-s)^{\alpha}} \ ds.
   \]
  Since $|s|^{\nu - 1}$ and $(t-s)^{-\alpha}$ are both increasing functions of $s$ for $s<0$ it follows that
   \[
     \int_{- \infty}^{t}\frac{\nu|s|^{\nu-1}}{(t-s)^{\alpha}} \ ds
   \]
  is an increasing function of $t$. If $t\leq -1$, then
   \[
     \partial_t^\alpha h
      \geq  \int_{- \infty}^{t}\frac{-\nu|s|^{\nu-1}}{(t-s)^{\alpha}} \ ds 
      \geq  \int_{- \infty}^{-1}\frac{-\nu|s|^{\nu-1}}{(-1-s)^{\alpha}} \ ds 
      \geq -c_{\alpha,\nu}.
   \]
  If $t>-1$, then
   \[
     \partial_t^\alpha h(t) \geq 
        \int_{-\infty}^{-1}\frac{-\nu|s|^{\nu-1}}{(t-s)^{\alpha}} \ ds 
      \geq  \int_{- \infty}^{-1}\frac{-\nu|s|^{\nu-1}}{(-1-s)^{\alpha}} \ ds 
      \geq -c_{\alpha,\nu}.
   \]  
 \end{proof}

\section{Discretization in time}  \label{s:discretize}
 To prove existence of solutions to \eqref{e:w} we will discretize in time. This discretization will also
 be useful when proving the H\"older continuity. This section contains properties of a discrete fractional
 derivative which we will utilize.
 
 To find a solution we subdivide the interval $(a,T)$ into $k$ intervals and let $\e = T/k$.
 For each fixed $k$ we may solve via recursion  
  \begin{equation}  \label{e:firstd}
   \begin{aligned}
   &\alpha \e
    \sum_{i <j} \frac{w(a+\e j,x) - w(a+\e i,x)}{(\e (j-i))^{1+\alpha}} \\
     &\quad = \int_{\R^n}[w(a+\e j,y)-w(a+\e j,x)]K(a+\e j,x,y)dy +  f(a+\e j,x),
   \end{aligned}
  \end{equation}
 for each $-\infty<j\leq k$. Here $f(t,x)=f(a,x)$ for $t<a$, so that $w(a+\e j,x)=w(a,x)$ for $j<0$. 
 
 For future reference we denote the discrete Caputo derivative as
 \begin{equation}  \label{e:dc}
  \partial_{\e}^{\alpha} u(a+\e j):=   \epsilon \alpha 
    \sum_{-\infty<i <j} \frac{u(a+\e j) - u(a+\e i)}{(\e (j-i))^{1+\alpha}}
 \end{equation}
 
 The following integration by parts type estimate will be useful 
 \begin{lemma} \label{l:decompose}
  \begin{equation}  \label{e:decompose}
   \begin{aligned}
    \alpha \sum_{0<j\leq k} \e u(a+\e j) \partial_{\e}^{\alpha} u(a+\e j)&\geq 
       \e^{1-\alpha} \mathop{\sum \sum}_{0\leq i < j \leq k} \frac{u^2(a+\e j) - u^2(a+\e i)}{(j-i)^{1+\alpha}} \\ 
     & \ + 
       \frac{\e^{1-\alpha}}{2} \sum_{0<j< k} \frac{u^2(a+\e j)}{ 2^{1+\alpha}(k-j)^{\alpha}} \\ 
     & \ + \frac{\e^{1-\alpha}}{2} \sum_{0<j\leq k} \frac{u^2(a+\e j)}{2 j^{\alpha}}  \\ 
      & \ - \e^{1-\alpha} \sum_{0<j\leq k}\frac{u(a)u(a+\e j)}{j^{\alpha}}.
    \end{aligned}
   \end{equation}    
 \end{lemma}

 \begin{proof}
  For notationally simplicity we assume throughout this proof that $a=0$. 
  For $i>0$ we write
   \[
    u(\e j)[u(\e j)- u(\e i)] = [u(\e j)- u(\e i)]^2 /2 + [u^2(\e j) - u^2(\e i)]/2.
   \]
  We note that 
   \[
     \sum_{2j-k \leq i < j} \frac{u^2(j)}{(j-i)^{1+\alpha}} = \sum_{j < i \leq k} \frac{u^2(j)}{(i-j)^{1+\alpha}},
   \]
  We thus conclude
   \begin{equation}  \label{e:break}
    \begin{aligned}
    \alpha  \sum_{0<j\leq k} \e u(\e j) \partial_{\e}^{\alpha} u(\e j)
     &= \alpha \e^{1-\alpha}\mathop{\sum \sum}_{0\leq i < j \leq k} \frac{[u(\e j) - u(\e i)]^2}{2(j-i)^{1+\alpha}} \\
     & \quad + \e^{1-\alpha}\sum_{0<j\leq k}  \sum_{-\infty<i<2j-k} \frac{u^2(\e j)}{2(j-i)^{1+\alpha}} \\
     & \quad + \e^{1-\alpha}\sum_{0<j\leq k}  \sum_{-\infty<i<j} \frac{u^2(\e j)/2-u(\e j)u(0)}{(j-i)^{1+\alpha}}.
     \end{aligned}
   \end{equation}
  We also have the following bound for $l<j$
   \begin{equation}  \label{e:lusual}
    \sum_{-\infty<i<l} (j-i)^{-(1+\alpha)} \geq 2^{-(1+\alpha)}\int_{-\infty}^{l} (j-t)^{-(1+\alpha)}
     \geq 2^{-(1+\alpha)} (j-l)^{1+\alpha}.
   \end{equation}
  Applying \eqref{e:lusual} to the appropriate terms in \eqref{e:break}, 
  and ignoring the  appropriate positive term when $j=k$ we obtain \eqref{e:decompose}.
 \end{proof}

 

 This next lemma is analogous to Lemma \ref{l:fracbound} for the discrete Caputo derivative. 
  \begin{lemma}  \label{l:dfracbound}
   Let $h$ be as in Lemma \ref{l:fracbound}. Then for $0<\epsilon < 1$ there exists $c_{\nu,\alpha}$
   depending on $\alpha$ and $\nu$ but 
   independent of $a$ such that 
    \[
     \partial_{\epsilon}^{\alpha} h(t) \geq -c_{\nu,\alpha}
    \] 
   for $t\in \epsilon \Z$ and $a<t<0$.
  \end{lemma} 
  
  \begin{proof}
   Since for $(\epsilon j)<-1$, $h(t)$ is a decreasing function we have 
    \[
     _{a}\partial_{\e}^{\alpha} h(t) \geq _{- \infty}\partial_{\e}^{\alpha} h(t).
    \]
   Recall that  
    \[
     _{- \infty}\partial_{\epsilon}^{\alpha} h(\epsilon j) 
      = \epsilon^{-\alpha} \alpha \sum_{-\infty <i < j}\frac{h(\epsilon j)-h(\epsilon i)}{(j-i)^{1+\alpha}}.
    \]
   Since $\nu<1$, if $t_1<t_2\leq 1$, then
    \[
     h(t_1)- h(t_1-h) \geq h(t_2) - h(t_2 - h). 
    \]
   Then $\partial_{\e}^{\alpha} h(\e j_2) \leq \partial_{\e}^{\alpha} h(\e j_1)$ if $\e j_1 < \e j_2\leq -1$. 
   As in Lemma \ref{l:fracbound} we also have 
   $\partial_{\e}^{\alpha} h(\e j_2) \geq \partial_{\e}^{\alpha} h(-1)$ and 
   $\partial_{\e}^{\alpha} h(-1)$ is uniformly bounded below for $0<\e<1$. 
   Then for $\epsilon j <0$ we conclude the lemma. 
  \end{proof}
  
 This next Lemma gives a fractional Sobolev bound for an extension of discrete functions. Throughout this 
 paper whenever we have a function $u$ defined on $\epsilon \Z$ we extend $u$ to all of $\R$ by
  \[
   u(t)= u(\epsilon j) \quad \text{  for  } j-1<t\leq j.
  \]
 This extension works particularly well for the Caputo derivative. 
 
%

 \begin{lemma}  \label{l:extend}
  If $u_{\e}$ is the appropriate extension of $u$, then there exists $c$ depending on $\alpha$, but independent of 
  $a$ such that if $a<-1$, then with $a+\epsilon k=T$
   \[
    \alpha \int_a^T \int_{-\infty}^t \frac{(u(t)-u(s))^2}{(t-s)^{1+\alpha}} \ ds \ dt 
      \leq c \epsilon^{2} \alpha \mathop{\sum \sum}_{ i<j\leq k} 
          \frac{(u(a+\e j)-u(a+\e i))^2}{(\e(j-i))^{1+\alpha}}    
   \]
  and  
   \[
    \int_a^T \frac{u(t)}{(t-a)^{\alpha}} \leq c \e \sum_{0<j\leq k} \frac{u^2(a+\e j)}{(a+\e j)^{\alpha}}.
   \]
 \end{lemma}
  
 \begin{proof}
  We note that for $i<j$
   \[
    \int_{\e(j-1)}^{\e j} \int_{\e(i-1)}^{\e i} \frac{1}{(t-s)^{1+\alpha}} \ ds \ dt
     \leq \e^2 \max \left\{ 2^{1+\alpha} , \frac{2-2^{1+\alpha}}{1-\alpha} \right\}
      \frac{1}{(\e(j-i))^{1+\alpha}}. 
   \] 
  The first conclusion then follows. 
  
  We now claim
   \begin{equation}  \label{e:alphab}
    \frac{1}{j^{\alpha}} \leq \int_{j-1}^j t^{-\alpha} \ dt  \leq \frac{1}{(1-\alpha)j^{\alpha}}.
   \end{equation}
  The first inequality is trivial. To prove the second inequality we compute
   \[
    \begin{aligned}
    \int_{j-1}^j t^{-\alpha} &= \frac{1}{1-\alpha} \left[j^{1-\alpha} -(j-1)^{1-\alpha}\right] \\
                             &= \frac{j^{-\alpha}}{1-\alpha} \left[j-(j-1)\left(\frac{j}{j-1}\right)^{\alpha} \right] \\
                             &\leq \frac{j^{-\alpha}}{1-\alpha}.
    \end{aligned}            
   \]
  This inequality implies the second conclusion. 
 \end{proof}


\section{First De Giorgi Lemma}  \label{s:lemma1}
 In this section we prove De Giorgi's first Lemma commonly known as the ``$L^2, L^{\infty}$'' estimate. 
 For a solution $u$ to the local equation this was proved using the cut-off $(u-M)_+$ for a constant $M$. 
 We will use the cut-off function $(w-\psi)_+$ where $\psi$ is defined as follows:
  \[
    \psi(x,t) := (|x|^{\sigma/2} -1)_+ + (|t|^{\alpha /2} -1)_+.
  \]
 We note that we will only utilize $\psi$ when $t\leq 0$. 
 We recall that $\sigma$ refers to the fractional order of the kernel $K$ with bounds depending on $\Lambda$.
 For any $L \geq 0$, we define
  \[
   \psi_{L}(x,t) = L + \psi(x,t).
  \]


 In the next two sections the proofs we present of the first and second De Giorgi Lemmas would 
 apply to a solution of \eqref{e:w} if we knew the cut-offs of $w$ were valid test functions. 
 Since this is not known a priori we prove the Lemmas for the sequence of approximating functions 
 $w_{\e}$ and obtain the results of the Lemmas for the solution $w$. 
 
 In the next two sections we will abuse notation for convenience and also to make the proofs more 
 transparent. We will write $w$ to mean a solution of \eqref{e:firstd} and assume that $\e$ is understood.
 We also extend $w$ by $w(t)=w(a+\e j)$ and $\psi(t)=\psi(a+\e j)$ for $a+\e(j-1)<t \leq a+\e j$. 
 
 We also recall that $Q:=(a,T) \times \R^n$. 
 \begin{lemma}  \label{l:degiorgi1}
  There exists a constant $\kappa_0 \in (0,1)$ depending only on $n,\sigma,\Lambda,\alpha$ 
  - but independent of $\epsilon$ and $a$- such that for any solution
  $w:[a,0] \times \R^n \hookrightarrow \R$ to \eqref{e:firstd}
  with $\|f\|_{L^{\infty}(Q)}\leq 1$ and $a\leq -1$, the following implication for $w$ holds true. 
  
  If it is verified that 
   \[
    \int_a^T \int_{\R^n} [w(t,x) - \psi]_+^2 \ dx \ dt \leq \kappa_0 , 
   \]
  and
   \[
    w(a,x) \leq \psi(a,x)+1/2 \text{   for all   } x \in \R^n ,
   \]
  then we have
   \[
    w(t,x) \leq \frac{1}{2} + \psi(x,t)
   \]
  for $(t,x) \in [a,T] \times \R^n$. Hence, we have in particular that $w \leq 1/2$ on $[-1,0] \times B_1(0)$.
 \end{lemma}
 
 \begin{proof}
  If $w$ is a solution to \eqref{e:firstd} we take  $\e [w-\psi_L]_+$ as a test function. We will only consider
  $L\leq 1/2$, so that the assumption on the initial condition will apply.  
  We add in $j$ and integrate
  over $\R^n$ to obtain
   \begin{equation}  \label{e:energyineq}
   \begin{aligned}
    &\int_{\R^n} \sum_{0<j\leq k}\e (w-\psi_L)_+(a+\e j,x) 
     \partial_{\e}^\alpha w(a+\e j, x) + \sum_{0<j\leq k} \e B[w,(w-\psi_L)_+](a+\e j,x) \\
     &= \int_{\R^n} \sum_{0<j\leq k}\e (w-\psi_L)_+(a+\e j,x)f(a+\e j,x)
   \end{aligned}
   \end{equation}
  We write $v=(w-\psi_L)_+$. We also define $v$ and $f$ for non-integer values as we did for $w$ and $\psi$. 
  Then
   \[
    \sum_{0<j\leq k} \epsilon B[w,v] = \int_a^T B[w,v], 
   \]
  and 
   \[
    \int_{\R^n} \sum_{0<j\leq k}\e (w-\psi_L)_+(a+\e j,x)f(a+\e j,x) = \int_{\R^n}\int_a^T fv.
   \]
  The elliptic portion is controlled exactly as in \cite{ccv11}, so that
   \begin{equation}  \label{e:elpart}
    \begin{aligned}
     &\int_{\R^n} \sum_{0<j\leq k}(w-\psi_L)_+(a+\e j,x) \partial_{\e}^\alpha w(a+\e j, x)
     +  \int_a^T\frac{1}{\Lambda}  \| v  \|_{H^{\frac{\sigma}{2}}(\R^n)}^2 \\
      &\quad \leq \quad C_{n,\Lambda , \sigma} \left[\int_a^t \int_{\R^n} v + v^2 + 
        \chi_{\{v>0\}}+ |f|v\right].
    \end{aligned}
   \end{equation}
  We now control the piece in time. We write $w=(w-\psi_L)_+ + (w-\psi_L)_- + \psi_L$.
  The term 
   \[
    (w-\psi_L)_+(a+\e j) \partial_{\e}^\alpha (w-\psi_L)_-(a+\e j) \geq 0,
   \]
  so we may ignore this term on the left hand side of the equation. We will however utilize it in a 
  crucial way in the second DeGiorgi lemma. We also recognize that $(w-\psi_L)_+(x,a)=0$ for all $x\in \R^n$ 
  by assumption, so that
   \[
    \sum_{0<j\leq k}\e v(a+\e j,x) \partial_{\e}^\alpha v(a+\e j, x)
     = \sum_{-\infty<j\leq k}\e v(a+\e j,x) \partial_{\e}^\alpha v(a+\e j, x).
   \]
  We move the term involving $\partial_{\epsilon}^{\alpha} \psi_L$ to the right hand side of the equation
  and use Lemma \ref{l:dfracbound} to control this term by the $L^1$ norm of $v$. We now utilize 
  Lemmas \ref{l:extend} and \ref{l:decompose} to conclude
   \[
    \begin{aligned}
     &\int_{\R^n} \int_a^T \frac{v^2(t,x)}{(t-a)^{\alpha}} + 
      \int_{\R^n} \int_a^t \int_a^t \frac{(v(t,x)-v(s,x))^2}{(t-s)^{1+\alpha}} \ ds \ dt 
      + \int_a^T\frac{1}{\Lambda}  \| v  \|_{H^{\frac{\sigma}{2}}(\R^n)}^2 \\
      &\leq \quad C_{n,\Lambda , \sigma, \alpha} \left[\int_a^t \int_{\R^n} v + v^2 + 
        \chi_{\{v>0\}} + |f|v\right].
    \end{aligned}
   \] 
  Since $|f|\leq 1$, the term $(|f|+1)v$ is controlled by $2v$. 
  Using Lemma \ref{l:ext}  we then conclude   
    \begin{equation}  \label{e:step11}
     \begin{aligned}
      &\int_{\R^n}  \| v \|_{H_{\alpha /2}(\R)}^2 
       + \int_a^T \frac{1}{\Lambda}  \| v  \|_{H^{\frac{\sigma}{2}}(\R^n)}^2 \\
      &\quad \leq  C_{n,\Lambda , \sigma, \alpha} \left[\int_a^t \int_{\R^n} v + v^2 + 
         \chi_{\{v>0\}} \right].
     \end{aligned}
    \end{equation}
   We now use the Sobolev embedding 
   $H^{\alpha /2}(\R) \subset L^{\frac{2}{1-\alpha}}(\R) $ to obtain from \eqref{e:step11} 
    \begin{equation}  \label{e:step2}
     \begin{aligned}
       &\int_{\R^n} \left( \int_{a}^T v^{\frac{2}{1-\alpha}} \right)^{1-\alpha}
            + \int_a^T  \left( \int_{\R^n} v^{\frac{2n}{n-\sigma}}   \right)^{\frac{n- \sigma}{n}} \\
             & \  \leq C_{n,\Lambda, \sigma , \alpha}
              \int_a^T \int_{\R^n} {v + v^2 + \chi_{ \{v>0\} }}. 
     \end{aligned}
    \end{equation}
   In the following computation we use Holder's inequality twice with
    \[
     \frac{\beta}{p_1} + \frac{1 - \beta}{p_2} = \frac{1}{p} = \frac{\beta}{p_3} + \frac{1 - \beta}{p_4}.
    \]
   We now interpolate as follows
    \[
     \begin{aligned}
      \int_a^T \int_{\R^n} v^p &=  \int_a^T \int_{\R^n} v^{p \beta} v^{ p (1-\beta)} \\
                          &\quad \leq \int_a^T \left( \int_{\R^n} v^{p_1}\right)^{\frac{p \beta}{p_1}}
                                           \left( \int_{\R^n} v^{p_2}\right)^{\frac{p (1-\beta) }{p_2}} \\
     &\quad \leq \left( \int_a^T \left( \int_{\R^n} v^{p_1}\right)^{\frac{p_3}{p_1}} \right)^{\frac{\beta p}{p_3}}
       \left( \int_a^T \left( \int_{\R^n} v^{p_2}\right)^{\frac{p_4}{p_2}} \right)^{\frac{p (1-\beta)}{p_4}}
     \end{aligned}
    \]
   so that 
    \[
     \left( \int_a^T \int_{\R^n} v^p \right)^{\frac{2}{p}}  \leq
     \beta \left( \int_a^T \left( \int_{\R^n} v^{p_1}\right)^{\frac{p_3}{p_1}} \right)^{\frac{2}{p_3}} +
      (1- \beta) \left( \int_a^T \left( \int_{\R^n} v^{p_2}\right)^{\frac{p_4}{p_2}} \right)^{\frac{2}{p_4}}.
    \]
   We now choose 
    \[
     p_1 = 2, \qquad p_2 = \frac{2n}{n - \sigma}, \qquad p_3 = \frac{2}{1- \alpha}, \qquad p_4 = 2
    \]
   so 
    \begin{equation}  \label{e:p}
     p=2\left(\frac{\alpha n+\sigma}{\alpha n+(1-\alpha)\sigma}\right) 
     \text{   and   } \beta =\frac{\sigma}{\alpha n + \sigma}. 
    \end{equation}
   By Minkowski's inequality 
    \[
     \left( \int_a^T \left( \int_{\R^n} v^2 \right)^{\frac{1}{1-\alpha}} \right)^{1-\alpha}
     \leq \int_{\R^n} \left( \int_a^T v^{\frac{2}{1-\alpha}}\right)^{1-\alpha}.
    \]
   Then
    \begin{equation}  \label{e:step3}
     \|  v \|_{L^{p}(\R^n \times [a,T])}^2 \leq C_{n,\Lambda , \sigma} \left[\int_a^t \int_{\R^n} v + v^2 + 
         \chi_{\{v>0\}} \right].
    \end{equation} 
   where $p>2$ is defined in \eqref{e:p}. 
   
   We now begin the Nonlinear recurrence. We let $L_K = \frac{1}{2}(1-2^{-k})$. 
   We also define
   \[
    U_k := \int_a^t \int_{\R^n} (w-\psi_{L_{k}})_+ + (w-\psi_{L_k})_+^2 + \chi_{ \{w-\psi_{L_k} >0\} }. 
         \chi_{\{v>0\}}
   \]
   Using Tchebychev's inequality and then \eqref{e:step3}, we have 
     \begin{alignat*}{1}
      \int_Q (w-\psi_{L_k})_+ &\leq \int_{Q}(w-\psi_{L_{k-1}})_+ 
                 \chi_{ \{w-\psi_{L_{k-1}} > \frac{1}{2^{k+1}}\} }  \\
                               &\leq (2^{k+1})^{p-1} \int_{Q}(w-\psi_{L_{k-1}})_+^{p} \\
                               &\leq (2^{k+1})^{p-1} C_n^p U_{k-1}^{p/2} \\
      \int_Q \chi_{ \{w-\psi_{L_k} >0\} } &\leq (2^{k+1})^p  \int_{Q}(w-\psi_{L_{k-1}})_+^{p} \\
                                &\leq (2^{k+1})^{p} C_n^p U_{k-1}^{p/2} \\
      \int_Q (w-\psi_{L_k})_+^2 &\leq \int_{Q}(w-\psi_{L_{k-1}})_+^2 
                 \chi_{ \{w-\psi_{L_{k-1}} > \frac{1}{2^{k+1}}\} }  \\  
                                &\leq (2^{k+1})^{p-2} \int_Q (w-\psi_{L_{k-1}})_+^{p}\\
                                &\leq (2^{k+1})^{p-2} C_n^p U_{k-1}^{p/2}.                     
     \end{alignat*}
   From the above three inequalities we conclude  
    \begin{equation}  \label{e:nonlinear}
     U_k \leq \overline{C}^k U_{k-1}^{p/2}, \qquad \text{ for every  } k \geq 0,
    \end{equation}
   for some universal constant $\overline{C}$ that depends only on $n,\Lambda,\sigma,\alpha$. 
   Since $p>2$ it follows from the nonlinear recurrence relation \eqref{e:nonlinear} 
   that there exists some sufficiently small
   constant $\kappa_0$ depending only on $n,\Lambda,\sigma,\alpha$ (but not $\epsilon$ or $a$) 
   such that if $U_1 \leq \kappa_0$, it 
   follows that $\lim_{k \to \infty} U_k \to 0$. From \eqref{e:step2} we have 
    \[
     U_1 \leq C\int_a^T \int_{\R^n}{|w-\psi|^2} \ dx \ dt.
    \]
   Furthermore, $U_k \to 0$ implies that 
    \[
     w \leq \psi + 1/2 \text{   for   } t \in [a,T] \times \R^n.
    \]  
 \end{proof}


 For this next corollary we define 
  \[
    \overline{\psi}(t,x):= 
       (|x|^{\sigma /4}-1)_+ + (|t|^{\alpha /4}-1)_+ \mbox{  if  } t\leq 0 
  \]
 
 \begin{corollary}  \label{c:rewrite}
  There exists a constant $\delta \in (0,1)$ and $\epsilon_0>0$, 
  both depending only on $n,\sigma,\Lambda,\alpha$ such that 
  for any solution $w: [a,0]\times \R^n \hookrightarrow \R$ to \eqref{e:firstd} with
  $\epsilon< \epsilon_0$, $a\leq -1$, and $\|f\|_{L^{\infty}((a,T)\times \R^n)}\leq 1$ satisfying
   \[
    w(t,x) \leq 1 + \overline{\psi}(t,x)  \text{   on  } [a,0] \times \R^n
   \]
  and
   \[
    | \{w>0\} \cap ([-2,0] \times B_2)|  \leq \delta
   \]
  we have
   \[
    w(t,x) \leq \frac{1}{2} \text{   for  } (t,x) \in [-1,0] \times B_1.  
   \]
 \end{corollary}
 
 \begin{proof}
  If $R \geq 2^{1/\sigma}$ and $R^{\sigma / \alpha}\geq 2^{1/\alpha}$, then 
   \[
    \begin{aligned}
     &1+((|y|+1)^{\sigma /4}-1)_+  \leq (|y|^{\sigma /2}) \quad \text{ if } |y| \geq R \\
     &1+((|t|+1)^{\alpha /4}-1)_+  \leq (|t|^{\alpha /2}) \quad \text{ if } |t| \geq R^{\sigma / \alpha}. 
    \end{aligned}
   \]
  Thus, we may choose $R$ even larger such that 
  \begin{equation} \label{e:Rbig}
   2+\overline{\psi}(|t|+1,|y|+1) \leq \psi(|t|,|y|) \text{  if  }  (y,t) 
   \notin [-R^\frac{\sigma}{\alpha},0] \times B_R(0).
  \end{equation} 
  $R$ is now fixed and is dependent on  $n, \sigma, \alpha$.

  For any
  $(t_0,x_0) \in [-1,0] \times B_1$ with $t_0 \in a+\epsilon \Z_+$
  , we introduce the rescaled function $w_R$ defined on 
  $[R^{\sigma / \alpha}(a-t_0) , R^{\sigma /\alpha} -t_0] \times \R^n$ by
   \[
    w_R(s,y) := w(t_0 + \frac{s}{R^{\sigma /\alpha}} , x_0 + \frac{y}{R}).
   \]
  The function $w_R$ satisfies equation \eqref{e:firstd} with intial time $R^{\sigma / \alpha}(a-t_0)$,
  with discrete time increment $\epsilon R^{\sigma / \alpha}$ and
  with a rescaled kernel
   \[
    K_R(t,x,y) := 
     \frac{1}{R^{n+\sigma}} K(t_0 + \frac{t}{R^{\sigma / \alpha}}, x_0 + \frac{x}{R}, x_0 + \frac{y}{R}).
   \]
  This rescaled kernel satisfies 
   \[
    \frac{1}{\Lambda} \chi_{ \{|x-y| \leq 3R \}} \frac{1}{|x-y|^{n+\sigma}}
     \leq K_R(t,x,y) \leq \Lambda \frac{1}{|x-y|^{n+\sigma}},
   \]
  which is a stronger hypothesis. 
  The right hand side of the equation is 
   \[
    f_R := f(t_0 +\frac{s}{R^{\sigma /\alpha}},x_0 + \frac{y}{R}),
   \]
  with $|f_R|\leq 1$. 
  We then choose $\epsilon_0 R^{\sigma / \alpha}=1$, 
  so that $\epsilon R^{\sigma / \alpha}<1$. We can apply Lemma \ref{l:degiorgi1} to $w_R$. In \cite{ccv11}
  it is shown that for $x_0 \in B_1$,  
   \[
    (|x_0 + x/R|^{\sigma / 4} -1)_+ \leq (|x_0 + x|^{\sigma / 4} -1)_+.
   \]
  which is also true in the one dimensional case, so 
   \[
    (|t_0 + \frac{t}{R^{\sigma / \alpha}}|^{\alpha / 4} -1)_+ \leq (|t_0 + t|^{\alpha / 4} -1)_+,
   \]
  and we conclude that $\overline{\psi}(t_0 + t/R^{\sigma / \alpha}, x_0 + x/R) \leq \overline{\psi}(t_0 + t, x_0 + x)$.
  Now since $\overline{\psi}(t,x)$ increases with respect to $|x|$ and $|t|$ for $|t|,|x|>1$ we have
   \[
    w_R(s,y) \leq 2 + \overline{\psi}(t_0 + \frac{s}{R^{\sigma / \alpha}}, x_0 + \frac{y}{R})  
     \leq 2 + \overline{\psi}(t_0 + s, x_0 + y) \leq 2 + \overline{\psi}(|s|+1,|y|+1).
   \]
  Then utilizing \eqref{e:Rbig} we have chosen $R$ large enough so that $w_R(s,y) \leq \psi(s,y)$ for any 
  $(s,y) \notin [-R^{\sigma / \alpha},0]\times B_R$.
   \[
    \begin{aligned}
     \int_{R^{\frac{\sigma}{\alpha}}(a-t_0)}^{0} \int_{\R^n} [w_R(s,y)- \psi(s,y)]_+^2
      &= \int_{-R^{\frac{\sigma}{\alpha}}}^0 \int_{|y| \leq R} [w_R(s,y)- \psi(s,y)]_+^2   \\
      &\leq \int_{-R^{\frac{\sigma}{\alpha}}}^0 \int_{|y| \leq R} [w_R(s,y)]_+^2  \\
      &= R^{n+ \frac{\sigma}{\alpha}} \int_{t_0 -1}^{t_0} \int_{\{x_0\}+B_1}[w(s,y)]_+^2 \\
      &\leq R^{n+ \frac{\sigma}{\alpha}} \int_{-2}^0 \int_{B_2} [w(s,y)]_+^2 \\
      &\leq R^{n+ \frac{\sigma}{\alpha}}(1+\overline{\psi}(-2,2))^2 \delta.
    \end{aligned}
   \]
  Choosing $\delta = R^{-n- \frac{\sigma}{\alpha}}(1+\overline{\psi}(-2,2))^{-2} \kappa_0$ gives that
  $w(t_0,x_0) \leq 1/2$ for $(t_0,x_0) \in (-1,0) \times B_1$ with $t_0 \in a + \epsilon \Z_+$, and
  therefore for all $(t,x) \in (-1,0) \times B_1$. 
 \end{proof}

\section{The Second De Giorgi Lemma}  \label{s:lemma2}
 For this section we will need the following functions
  \[
   F_1(x) := \sup (-1, \inf (0,|x|^2 -9)) \qquad F_2(t):= \sup (-1, \inf (0,|t|^2 -16)).
  \]
 We note that $F_i$ are both Lipschitz. $F_1$ is compactly supported in $B_3$ and equal to $-1$ in $B_2$. 
 Similarly, $F_2$ is compactly supported in $[-4,4]$ and is equal to $-1$ in $[-3,3]$. We will only use $F(t)$
 for $t \leq 0$. 
 
 For $\lambda < 1/3$, we define 
  \[
   \psi_{\lambda}(t,x) :=
      ((|x|- \lambda^{-4/ \sigma})^{\sigma /4}-1)_+   \chi_{\{|x|\geq \lambda^{-4/\sigma} \}} +
      ((|t|- \lambda^{-4/ \alpha})^{\alpha /4}-1)_+   \chi_{\{|t|\geq \lambda^{-4/\alpha} \}}. 
  \]  
 Our Lemma will involve the following sequence of five cutoffs:
  \[
   \phi_i := 2+ \psi_{\lambda^3}(t,x)  + \lambda^iF_1(x) + \lambda^iF_2(t) 
  \]
 for $0 \leq i \leq 4$. 
 
 \begin{lemma}  \label{l:second}
  Let $\delta$  be the constant defined in Corollary \ref{c:rewrite}. 
  For $0<\mu <1/8$ fixed, there exists $\lambda \in (0,1)$, depending only on $n,\Lambda, \sigma,\alpha$  
  such that for any solution $w:[a,0]\times \R^n \to \R$ to \eqref{e:firstd} with
  $0< \e <\e_0$, $a\leq -4$, and $|f|\leq 1$satisfying
     \[
      \begin{aligned}
      w(t,x) &\leq 2 + \psi_{\lambda^3}(t,x) \text{  on  } [a,0] \times \R^n \\
      &|\{w<\phi_0\} \cap ((-3,-2) \times B_1)| \geq \mu,
     \end{aligned}
     \]
  then we have 
   \[
    |\{w>\phi_4\} \cap ((-2,0) \times B_2)| \leq \delta.
   \]
 \end{lemma}
 The main idea of the proof is as follows: We utilize that $w$ satisfies the equation as well as
 the nonlocal nature of the spatial operator to make use of the assumption that the set where $w$ is small
 on $(-3,-2)\times B_1$ is significant enough, and show there is a large set of points in $(-3,-2)\times B_2$ in which
 $w$ is small. We then use the nonlocal nature of the time derivative to show that this implies that the set in 
 $(-2,0) \times B_2$ on which $w$ is large is a very small set. 
  
 \begin{proof}
  Throughout this proof the constants $c,C$ will denote constants that
  only depend on the parameters $n,\sigma, \alpha, \Lambda$ and nothing else. They can change from line to 
  line in the proof. We first consider $0<\lambda<1/3$ and small enough such that 
   \[
    \psi_{\lambda}(t,x) =0 \quad \text{  if  } (t,x) \in [-2,0] \times B_2.
   \]
  We split the proof into three steps. 
  
  \textbf{First Step: The energy inequality.} We return to the energy inequality \eqref{e:energyineq}, but this
  time we utilize the two positive terms that were previously ignored. We have for $v=(u-\phi_1)_+$
   \[
    \begin{aligned}
      & \int_{\R^n} \int_a^T (w-\phi_1)_+ \partial_{\e}^{\alpha} [(w-\phi_1)_+ 
          + (w-\phi_1)_- 
          + \phi_1   ] dt dx \\ 
      &\quad \int_a^T B((w-\phi_1)_+,(w-\phi_1)_+) + 
        B((w-\phi_1)_+,\phi_1) + B((w-\phi_1)_+,(u-\phi_1)) dt \\
      &\leq \int_a^T \int_{\R^n} (w-\phi_1)_+ |f|.
    \end{aligned}
   \]
  Since $(w-\phi_1)_+$ is compactly supported in $[-4,0] \times B_3$, 
   \[
   \int_{Q} |f| (w-\phi_1)_+ \leq C\lambda
   \]
  The spatial pieces involving $B$ are controlled as follows as in \cite{ccv11}:
   \[
    \begin{aligned}
    &|B((w-\phi_1)_+,\phi_1)| \leq \frac{1}{2}B((w-\phi_1)_+,(w-\phi_1)_+)   \\
            &\qquad + 2\int \int [\phi_1(x)-\phi_1(y)]^2K(x,y)[\chi_{B_3(x)}].
    \end{aligned}
   \]
  The first term is absorbed on the left hand side and
   \[
    2\int \int [\phi_1(x)-\phi_1(y)]^2K(x,y)[\chi_{B_3(x)}] \leq C\lambda^2.
   \]
  Then our inequality becomes
   \[
    \begin{aligned}
      C\lambda^2 &\geq \int_{\R^n} \int_a^T (w-\phi_1)_+ \partial_{\epsilon}^{\alpha} (w-\phi_1)_+ dt dx \\
       &\quad + \int_{\R^n} \int_a^T (w-\phi_1)_+ \partial_{\epsilon}^{\alpha} (w-\phi_1)_- dt dx \\
       &\quad + \int_{\R^n} \int_a^T (w-\phi_1)_+ \partial_{\epsilon}^{\alpha} \phi_1    dt dx \\ 
       &\quad \int_a^T \frac{1}{2}B(v,v) + B(v,(u-\phi_1)_-) dt.
    \end{aligned}
   \]
  The time piece is controlled as follows:
  As shown in the first De Giorgi Lemma \ref{l:degiorgi1} we may use the Sobolev embedding to obtain
   \[
    \int_{\R^n} \int_a^T (u-\phi_1)_+ \partial_{\epsilon}^{\alpha} (u-\phi_1)_+ 
    \geq c \int_{\R^n} \left(\int_a^T (u-\phi_1)_+^p \right)^{2/p}.
   \]
  By utilizing that $(u-\phi_1)_+$ is compactly supported in the time interval $[-4,0]$ we have the bound
   \begin{equation}  \label{e:sob1}
    \int_{\R^n} \int_a^T (u-\phi_1)_+ \partial_{\e}^{\alpha} (u-\phi_1)_+ 
    \geq c \int_{\R^n} \int_a^T (u-\phi_1)_+^2 .
   \end{equation}
  We now control the other time piece by moving it to the right hand side of the equation and showing:
   \begin{equation}
    0\leq -\int_{\R^n} \int_a^T (u-\phi_1)_+ \partial_{\e}^{\alpha} \phi_1 
     \leq \int_{\R^n} \int_a^T c_1 (u-\phi_1)_+^2 + \frac{1}{4 c_1} (\partial_{\e}^\alpha \phi_1)^2 dt dx.
   \end{equation}
  By choosing $c_1$ appropriately we absorb the first term on the left hand side of the equation 
  by using \eqref{e:sob1}. 
  Now we also have for $t \in [a,0]$ with $t\in a +\epsilon \Z_+$,
    \[
     \begin{aligned}
      0 &\leq - \partial_{\epsilon}^{\alpha} \tilde{\psi}_{\lambda^3} \leq C \lambda^3 \\
      0 &\leq - \partial_{\epsilon}^{\alpha} \lambda^i F_2(t) \leq C\lambda^i.
     \end{aligned}
    \] 
  Thus, 
   \[
    \begin{aligned}
    &\int_{\R^n} \int_a^T \left( \frac{1}{2} (u-\phi_1)_+ \partial_{\e}^{\alpha} (u-\phi_1)_+ 
      +(u-\phi_1)_+ \partial_{\e}^{\alpha} (u-\phi_1)_- \right) dt dx \\
    &\quad +\int_a^T \frac{1}{2}B((u-\phi_1)_+,(u-\phi_1)_+) +B((u-\phi_1)_+,(u-\phi_1)_-)\leq C \lambda^2  
    \end{aligned}
   \]
  and so
   \begin{equation}  \label{e:good}
    \int_{\R^n} \int_a^T (u-\phi_1)_+ \partial_{\e}^{\alpha} (u-\phi_1)_-  dt dx 
      +\int_a^T B((u-\phi_1)_+,(u-\phi_1)_+) dt  \leq C \lambda^2  
   \end{equation}

  \textbf{Second step: An estimate on those time slices where the ``good'' spacial term helps.} We recall
   that $\mu <1/8$ is fixed from the beginning of the proof. From our hypothesis,
    \[
     |\{w<\phi_0\} \cap ((-3,-2) \times B_1)| \geq \mu,
    \]
   the set of times $\Sigma$ in $(-3,-2)$ for which $|w(\cdot, T)<\phi_0| \geq \mu/4$ has at least measure
   $\mu/(2|B_1|)$. As in \cite{ccv11} we obtain
    \[
     \int_\Sigma \int_{\R^n} (w-\phi_1)_+ dx dt \leq C\Lambda \lambda^2.
    \]
   Now 
    \[
     \{w-\phi_2>0\}\cap(\Sigma \times B_2) \subset \{w-\phi_1>\lambda/2\} \cap (\Sigma \times B_2),
    \] 
   and so from Tchebychev we have
    \[
     |\{w>\phi_2\} \cap (\Sigma \times B_2)| \leq C\lambda.
    \]
   which we rewrite as  
    \begin{equation}  \label{e:good2}
     | \{w \leq \phi_2\} \cap (\Sigma \times B_2)| 
      \geq |\Sigma \times B_2| - C\lambda = \mu /2 - C\lambda
    \end{equation}
   This will be positive for $\lambda$ chosen small enough. This will only depend on $\mu,n,\sigma,\alpha$.

  \textbf{Third Step: Utilizing the extra good piece in time.} 
  We now utilize the second ``good'' extra term in time. 
  Since we chose $\psi_{\lambda^3}$  in place of $\psi_{\lambda}$  
  in the definition of $\phi$, then substituting $\phi_3$ in the place of 
  $\phi_1$ in \eqref{e:good} we obtain
   \begin{equation}  \label{e:l6}
    \int_{\R^n} \int_a^T (u-\phi_3)_+ \partial_{\epsilon}^{\alpha} (u-\phi_3)_-  dt dx 
      +\int_a^T B((u-\phi_3)_+,(u-\phi_3)_+) dt  \leq C \lambda^6.  
   \end{equation}
  We define the set
   \[
    A:= \{x \in B_2 :  |(x_0 \times \Sigma) \cap \{w\leq \phi_2\}| \geq |\Sigma|/2\}.
   \]
  Then from \eqref{e:good2} we obtain 
   \[
    |A|\geq |B_2| - \frac{C\lambda}{2|\Sigma|} \geq |B_2|(1 - C\lambda \mu^{-1}) .
   \]
  We choose $\lambda$ small enough such that 
  $|B_2|C\lambda \mu^{-1} \leq \delta /4$, so that
  \begin{equation} \label{e:A}
   |B_2 \setminus A| \leq \delta /4,
  \end{equation} 
  where $\delta$ is the constant in the statement of the theorem. 
  Recalling that for $x_0 \in A$
   \[
    |\{w(x_0,t) \leq \phi_2 \} \cap (x_0 \times [-3,-2])| \geq \frac{\mu}{2|B_2|},
   \]
  then for $x_0 \in A$ and $t_0\in[-2,0]$ and $t_0 \in a +\epsilon \Z_+$ with $(u-\phi_3)(x_0,t_0)>0$
   \[
    \begin{aligned}
     \partial_{\epsilon}^{\alpha} (w-\phi_3)_-(t_0) 
      &\geq \alpha \epsilon \sum_{0\leq i<j} \frac{-(\phi_3-w)(\epsilon i)}{(\epsilon(k-i))^{1+\alpha}} \\
      &\geq \alpha \int_{[-3,-2]\cap \{w>\phi_3\}} 
         \frac{\lambda^2}{|3|^{1+\alpha}}ds \\
      &\geq \alpha \mu /(2|B_1|) 
         \frac{\lambda^2}{|3|^{1+\alpha}} \\
      &\geq c\mu \lambda^2
    \end{aligned}
   \]
  Putting the above inequality together with \eqref{e:l6}:
   \[
    \begin{aligned}
     C\lambda^6 &\geq \int_{\R^n} \int_a^T (w-\phi_3)_+ \partial_{\epsilon}^{\alpha} (w-\phi_3)_- \\
      &\geq \int_{A} \int_{-2}^0 (w-\phi_3)_+ \partial_{\e}^{\alpha} (w-\phi_3)_- \\
      &\geq \int_{A} \int_{-2}^0 (w-\phi_3)_- c\mu \lambda^2
    \end{aligned}
   \]
  and thus 
   \[
    \int_{A} \int_{-2}^0 (w-\phi_3)_+ \leq C  \frac{\lambda^4}{\mu}.  
   \]
  We utilize Tchebyschev one more time to get
   \[
    \begin{aligned}
     |\{w-\phi_4>0\}\cap (A \times [-2,0])| &\leq |\{w-\phi_3>\lambda^3 \} 
       \cap (B_2\times [-2,0])|  \\
                                            &\leq 
                               \lambda^{-3} \int_{A \times [-2,0]} {(w-\phi_3)_+} \\
                                            &\leq C\frac{\lambda}{\mu},
    \end{aligned}
   \]
  so we finally obtain
   \[
    |\{w-\phi_4>0\}\cap (A \times [-2,0])| \leq C\frac{\lambda}{\mu}.
   \]
  We choose $\lambda$ small enough so that $C\frac{\lambda}{\mu} < \delta/2$. 
  Combining this estimate with \eqref{e:A} we have the desired inequality: 
   \begin{equation}   \label{e:measbound}
    |\{w>\phi_4\} \cap (B_2 \times [-2,0]) | \leq \delta.
   \end{equation}
 \end{proof}

\section{Proof of the Holder regularity}  \label{s:Holder}
 In this section we prove our main result. Since De Giorgi's lemmas were proven independent of $\e$, the 
 conclusions hold in the limit. Therefore we may prove the results for solutions to \eqref{e:w}; 
 however,
 the proofs can be given for analogous results of the discretized solutions to \eqref{e:firstd} (see
 Remark \ref{r:mod}).
 For $\lambda$ as in the previous section we define
  \[
    \psi_{\tau, \lambda}(t,x) :=
      ((|x|- \lambda^{-4/ \sigma})^{\tau}-1)_+   \chi_{\{|x|\geq \frac{1}{\lambda^{4/\sigma}} \}}  +
      ((|t|- \lambda^{-4/ \alpha})^{\tau}-1)_+   \chi_{\{|t|\geq \frac{1}{\lambda^{4/\alpha}} \}}. 
  \]  

 \begin{lemma}  \label{l:decrease}
  There exists $\tau_0$ and $\lambda^*$ such that if for any solution to \eqref{e:w} in 
  $[a,0] \times \R^n$ with $|f|\leq \lambda^4$ and $a\leq -4$ such that if
   \[
    -2-\psi_{\tau, \lambda}  \leq w \leq 2+\psi_{\tau, \lambda} ,
   \]
  we have 
   \[
    \sup_{[-1,0] \times B_1} w - \inf_{[-1,0]\times B_1} w \leq 4 - \lambda^*.
   \]
 \end{lemma}

 \begin{proof}
  We fix $\tau>0$ depending on $\lambda,\sigma,\alpha$ such that
   \[
    \frac{(|x|^\tau -1)_+}{\lambda^4} \leq (|x|^{\sigma /4}-1)_+  
    \quad \text{  and  } \quad \frac{(|t|^\tau -1)_+}{\lambda^4} \leq (|t|^{\alpha /4}-1)_+.
   \]
  Without loss of generality we assume that
   \[
    |\{w<\phi_0\} \cap ((-3,-2) \times B_1)| > \mu. 
   \]
  Otherwise the inequality is verified by $-w$. From Lemma \ref{l:second} 
   \[
    |\{w>\phi_4\} \cap ((-2,0)\times B_2)| \leq \delta.
   \]
  We define $\overline{w}(t,x):= \lambda^{-4}(w-2(1-\lambda^4))$.
  Since $\lambda$ was chosen so that 
   \[
    \phi_4(t,x)=2-2\lambda^4 \text{  for  } (t,x) \in [-2,0] \times B_2,
   \]
  we have 
   \[
    |\{\overline{w}>0\} \cap ([-2,0]\times B_2)| \leq 
    |\{w>\phi_4\} \cap ((-2,0)\times B_2)| \leq \delta.
   \]
  Also,
   \[
    \overline{w} \leq 2+ \frac{\psi_{\tau,\lambda^3}(x,t)}{\lambda^4} \leq 2+\psi_{\lambda^3} \leq 2 \psi_1
    \leq 2+\overline{\psi}.
   \]
  Furthermore, $\overline{w}$ satisfies \eqref{e:w} with right hand side $|f|\leq 1$. 
  Then we may apply Corollary \ref{c:rewrite} to $\overline{w}$, and conclude $\tilde{w}\leq 1/2$ on 
  $(-1,0) \times B_1$. Hence
   \[
    w(t,x) \leq 2 - \frac{3}{2}\lambda^4  \text{  for  }  (t,x) \in [-1,0] \times B_1.
   \]
 \end{proof}

 We now prove the Holder regularity
  \begin{theorem}  \label{t:holder}
   Let $w$ be a solution to \eqref{e:w} with $f \in L^{\infty}$. Then $w$ is Holder continuous. 
  \end{theorem}

  \begin{proof}
   Let $(t_0,x_0) \in (a , \infty) \times \R^n$. We assume that $(t_0 -a)>4$, otherwise we may rescale and
   have a new norm depending on the rescaling. We translate to the origin by considering
    \[
     w_0(t,x):= w(t_0 + t, x_0 + x). 
    \]
   Now we consider $\gamma<1$ such that
    \[
     \frac{1}{1-(\lambda^* /2)} \psi_{\tau,\lambda}(x) \leq \psi_{\tau,\lambda}(x),
     \text{  for  }  |x|\geq 1/\gamma.
    \]
   $\gamma$ only depends on $\lambda, \lambda^*, \tau$. We define by induction:
    \[
     \begin{aligned}
      &w_1(t,x) = \frac{w_0(t,x)}{\| w_0 \|_{L^{\infty}}+ \lambda^4\| f \|_{L^{\infty}}} , 
           (t,x) \in (a,0) \times \R^n, \\
      &w_{k+1}(t,x) = \frac{1}{1-\lambda^* /4}
      (w_k(\gamma^{\sigma / \alpha}t,\gamma x)-\overline{w}_k)
      , (t,x) \in (a\gamma^{-\sigma k},0) \times \R^n,
     \end{aligned}
    \]
   where
    \[
     \overline{w}_k := \frac{1}{|B_1|} \int_{-1}^0 \int_{B_1} w_k(t,x) dx dt.
    \]
   By construction, $w_k$ satisfies the hypothesis of Lemma \ref{l:decrease} for any $k$. So we conclude
    \[
     \sup_{t_0 +(-\gamma^{\sigma / \alpha},0) \times (x_0 + B_{\gamma^k})} w 
     -  \inf_{t_0 +(-\gamma^{\sigma / \alpha},0) \times (x_0 + B_{\gamma^k})} w
      \leq C(1-\lambda^* /4)^k,
    \]
   We then conclude that $w$ is $C^{\beta}$ with 
    \begin{equation}  \label{e:beta}
     \beta = \frac{\ln (1-\lambda^* /4)}{\ln \gamma^{\sigma / \alpha}} .
    \end{equation}
  \end{proof}

 \begin{remark}  \label{r:mod}
  The same methods work in this section to prove that if $w_{\e}$ is a solution to \eqref{e:firstd},
  then
   \begin{equation}  \label{e:mod}
    | w_{\e}(x,t)-w_{\e}(y,s) | \leq C\left( |t-s|^{\beta} + |x-y|^{\beta} \right)
   \end{equation}
  for $|t-s|,|x-y|>\rho(\e)$ where $\rho(\e)$ is a modulus of continuity with $\rho(\e) \to 0$
  as $\e \to 0$. The constant $C$ will be dependent on the distance from $t$ and $s$ to the initial point $a$. 
 \end{remark}

\section{Higher Regularity in Time}  \label{s:highertime}
In this section we show higher regularity in time if the right hand side $f$ is regular. 
A well established method - which we will follow - is to apply Theorem 
\ref{t:holder} to the difference quotient of a solution $w \in C^{0,\beta}$ in time. 
It will then follow that $w \in C^{0,2\beta}$ in time. Applying this method a finite number of times we will
obtain that $w \in C^{1,\beta}$ in time. See \cite{cc95} for an illustration of this method. 
There are two issues to resolve
before applying this method. The first is a translation of $w$ has a different initial point; therefore,
the difference of two functions does not satisfy an equation of type \eqref{e:w}. The second issue is seen
by formally differentiating the equation \eqref{e:first}
to obtain that $w'$ satisfies
\[
 ^{RL}D_t^{\alpha} w'(t,x) - \mathcal{L}(w'(x,t)) = f'(t).
\]
an equation involving the Reimann-Liouville derivative, so it is not obvious that one may continue
this iteration process past $C^{1,\beta}$. Both issues are overcome by considering $w(x,t)\eta(t)$ where
$\eta(t)$ is a smooth cut-off function in time that is zero until some point after the initial time. 
$w\eta$ will satisfy \eqref{e:w} and/or \eqref{e:firstd} 
with a different right hand side. 
Furthermore, $_{-\infty}\partial_t^{\alpha}(w\eta)= _{a}\partial_t^{\alpha}(w\eta)$. Then any translation
of $w\eta$ will have the same initial starting point at $-\infty$. Also, the Reimann-Liouville and Caputo derivatives
are the same at $-\infty$ if $w\eta(t)\to 0$ as $t \to -\infty$. 
 
 Given Theorem \ref{t:holder} it is still not immediate that $u \in H^{\alpha}$; 
 therefore, we may not assume that $(w-\psi)_+$ is an 
 acceptable test function. As before we must then approximate with the step functions. 
 
 In this section we utilize a cut-off function $\eta \in C^{\infty}$ with the following properties: 
 $\eta(t)$ is increasing, $\eta(t)=0$ for $t<1/2$, $\eta(t)=1$ for $t>1$.   
 \begin{lemma}  \label{l:step0}
  Let $w$ be a solution to \eqref{e:firstd} in $[0,1]\times \R^n$ with $f \in C^{\beta}$ and initial point $a=0$. 
  Then $w\eta$ is a solution
  to \eqref{e:firstd} in $(\infty, T)$ with right hand side $\tilde{f}\in L^{\infty}$ satisfying,
   \[
    |\tilde{f}(t) - \tilde{f}(s)| \leq C |t-s|^\beta
   \]
  as long as $|t-s|> \rho(\e)$ for $\rho(\e)$ as in Remark \ref{r:mod}. 
 \end{lemma}
 
 \begin{proof}
  We note that 
   \[
    w(\e j) \eta(\e j) - w(\e i) \eta(\e i) 
     = \eta(\e j) \left(w(\e j) - w(\e i)\right) + w(\e i) \left(\eta(\e j) - \eta (\e i) \right).
   \]
  From the mean value theorem
   \[
    \frac{\eta(\e j) - \eta (\e i)}{(\e(j-i))^{1+\alpha}} = \frac{\tilde{\eta}(\e i)}{(\e(j-i))^{\alpha}},
   \]
  with $\tilde{\eta}$ a smooth function. Combining the above two inequalities we obtain
   \[
    _{-\infty} \partial_{\e}^{\alpha} (w\eta)(\e j) = \eta(\e j)\partial_{\e}^{\alpha}w(\e j) 
     + \e \sum_{i<j}\frac{\tilde{\eta}(\e i) w(\e i)}{(\e(j-i))^{\alpha}}. 
   \] 
  Then $w\eta$ satisfies \eqref{e:firstd} with right hand side
   \[
    \tilde{f}= f - \e \sum_{i<j}\frac{\tilde{\eta}(\e i) w(\e i)}{(\e(j-i))^{\alpha}}. 
   \]
  Now for $\e h > \rho(\e)$ 
   \[
    \begin{aligned}
    &\frac{\e}{h^{\beta}} \sum_{0<i<j+h}\frac{\tilde{\eta}(\e i) w(\e i)}{(\e(j+h-i))^{\alpha}} 
     - \frac{\e}{h^{\beta}} \sum_{0<i<j}\frac{\tilde{\eta}(\e i) w(\e i)}{(\e(j-i))^{\alpha}} \\
     &= \e \sum_{0<i<j}\frac{\tilde{\eta} w(\e (j+h-i)) - \tilde{\eta}w(\e(j-i))}{h^{\beta}(\e i)^{\alpha}} 
      + \e \sum_{j\leq j+h} \frac{\tilde{\eta} w(\e (j+h-i))}{h^{\beta}(\e i)^{\alpha}} 
     &\leq C
    \end{aligned}
   \]
  where $C$ depends on the modulus of continuity of $w$ from Remark \ref{r:mod} but is independent of $\e$. 
  In fact the very last term goes to zero as $\e \to 0$ since $\psi(t)=0$
  for $t\leq 1/2$. The last inequality proves the Lemma.  
 \end{proof}
 
 In the next Lemma we assume for simplicity and transparency of the proof that the kernel $K$ is time-independent.
 However, this assumption is not necessary if $K$ is sufficiently smooth in time. 
 
 \begin{lemma}   \label{l:step1}
  Let $w$ be a solution to \eqref{e:firstd} in $[0,T]\times{\R^n}$ with right hand side $f \in C^{\beta}$. There
  exists a modulus of continuity $\rho(\e)$ with $\rho(\e) \to 0$ as $\e \to 0$, such that
  if $|t-s|>\rho(\e)$, then 
   \begin{equation}  \label{e:step1}
    |w(t,x)- w(s,x)| \leq C|t|^{2\beta}
   \end{equation}
  where $C$ depends on the distance from $t$ and $s$ to the initial point $0$. 
 \end{lemma}

 \begin{proof}
  We consider the difference quotient
   \[
    v:=\frac{\eta w(\e(j+h)) - \eta w(\e j)}{(\e h)^\beta}.
   \]
  From Remark \ref{r:mod} $v \in L^{\infty}$. From Lemma \ref{l:step0} $v$ satisfies \eqref{e:firstd}
  with an $L^{\infty}$ right hand side. Then from Remark \ref{r:mod} $v$ satisfies \eqref{e:mod}. 
  Then from \cite{cc95} $v$ satisfies \eqref{e:step1} with modulus $2\rho(\e)$. To obtain higher
  H\"older continuity closer to the initial time we simply scale and use the 
  alternative cut-off $\eta(Mt)$ with $M>1$. 
 \end{proof}
 
 We now show that like the heat equation there is instantaneous smoothing in time away from the initial
 start time. 
 \begin{theorem}  \label{t:smooth}
  Let $w$ be a solution to \eqref{e:w} in $[0,T]\times \R^n$ with right hand side $f \in C^{k,\beta}$, and kernel
  $K(x,y)$ independent of time. Then 
  $w \in C^{k,2\beta}$ away from the initial time $0$. 
 \end{theorem}
 
 \begin{proof}  
  As in \cite{cc95} we may apply Lemma \ref{l:step1} inductively  to the difference quotient
   \[
    v:=\frac{\eta w_{\e}(\e(j+h)) -\eta w_{\e}(\e j)}{(\e h)^{k\beta}}
   \]
  to the approximations $w_{\e}$.,
  to obtain the required ``$C^{0,(k+1)\beta}$'' regularity for $v$. We may do this finitely many times 
  to obtain that 
   \[
    |\eta w_{\e}(t,x) - \eta w_{\e}(s,x)| \leq C|t-s|
   \]
  with $C$ independent of $\e$, and for $|t-s|>\rho(\e)$ for some new modulus of continuity 
  with $\rho(\e)\to 0$ as $\e \to 0$. Then letting $\e \to 0$ we obtain $w$
  is ``Lipshitz'' and hence cut-offs of $\eta w$ are valid test functions. 
  Then we may apply the method in Lemma \ref{l:step1} directly to 
   \[
    v= \frac{\eta w(t+h,x) - \eta w(t,x)}{h},
   \] 
  and conclude $w \in C^{1,\beta}$ and $w$ satisfies \eqref{e:firstd} with right hand side $f'$. 
  Since $w\eta$ is actually a strong solution - in time - to \eqref{e:e} then the difference quotients will also 
  be such solutions. Then $w'$ is also a solution to \eqref{e:e}, and so we may utilize cut-offs of
  $w'$ as test functions and continue with the bootstrapping process. 
 \end{proof}
 
 In order to show that a solution $w$ to \eqref{e:w} is a strong - in time - solution, we require some 
 regularity of $w(0,x)$ to ensure that $w$ is continuous up to the initial time.
 
 \begin{lemma}  \label{l:initime}
  Let $w$ be a solution to \eqref{e:w} in $Q=[0,T]\times \R^n$. 
  Let $w(0,x)\in C^{0,\sigma}(\R^n) \ (C^{2\sigma -1}(\R^n))$ if $\sigma \leq 1 \ (\sigma>1)$. 
  Then $w \in C^{0,\beta}([0,T]\times \R^n)$ with $\beta$ as defined in \eqref{e:beta}. 
 \end{lemma}
 
 \begin{proof}
  We extend $w$ and $K$ backwards in time as $w(x,t)=w(0,t)$ 
  and $K(t,x,y)=K(0,x,y)$ if $t<0$. Then $w(x,t)$ is a solution
  to \eqref{e:w} in $[-5,T]\times \R^n$ with new right hand side
   \[
    F(x,t)=f(x,t)+ \chi_{\{-5\leq t\leq 0\}}\int_{\R^n}K(0,x,y)[w(0,x)-w(0,y)] \ dy
   \]
  which is in $L^{\infty}(\R^n)$ by the regularity assumptions on $w$ and also 
  since the original right hand side $f \in L^{\infty}$. Then by Theorem \ref{t:holder}
  we conclude the result. 
 \end{proof}
 
 \begin{corollary}  \label{c:strong}
  Let $w$ be a solution to \eqref{e:w} with kernel $K$ time-independent and right hand side $f \in C^{0,\alpha}$.
  Assume further that $w(0,x)\in C^{0,\sigma} \ (C^{2\sigma-1})$ if 
  $\sigma \leq 1 \ (\sigma>1)$. Then $w$ is a strong solution in time, i.e. $w$ is a solution to \eqref{e:e}.
 \end{corollary}
 
 \begin{proof}
  From Theorem \ref{t:smooth} and Lemma \ref{l:initime}, $w \in H^{\alpha}$ away from the initial time $0$ 
  and is continuous up to the initial time. Therefore we compute as before
  \[
   \begin{aligned}
   &\int_{\R^n} \phi \partial_t^{\alpha} (w \eta) + \mathcal{B}(w \eta , \phi) =
     \eta(t) \int_{\R^n} f(t,x) \phi(t,x) \\
     &\quad  + \eta(t) \int_{\R^n} \frac{\eta(t)w(0,x)}{t^{\alpha}}
     + \alpha \int_{\R^n} \int_0^t \frac{w(s,x)[\eta(t,x)-\eta(s,x)]}{(t-s)^{1+\alpha}} \ ds \ dt \ dx. 
   \end{aligned}
  \]
  We let $\eta(t) $ approach the heavy-side function by scaling $\eta(Mt)$ and letting $M \to \infty$. 
  By the continuity of $w$ up to the initial time $0$, the last term goes to zero as $M \to \infty$. 
  Moving the term involving $w(0)$ to the left hand side and integrating over $[0,T]$
  we obtain \eqref{e:e}.
 \end{proof}
 
 We may also answer a question of uniqueness. 
 \begin{corollary}  \label{c:unique}
  Let $w_1,w_2$ be two solutions to \eqref{e:w} with the same assumptions as in Corollary \ref{c:strong}.
  Assume further that $w_1(0,x)\equiv w_2(0,x)$. Then $w_1(t,x)\equiv w_2(t,x)$. 
 \end{corollary}
 
 \begin{proof}
  $v=w_1 - w_2$ is a solution. Furthermore, we may use $v=w_1 - w_2$ as a test function. Then from 
  Proposition \ref{p:changev}
   \[
    \begin{aligned}
     &\int_{\R^n}\int_{0}^T {v^2(t) \left[\frac{1}{(T-t)^{\alpha}} + \frac{1}{t^{\alpha}} \right]} \ dt \ dx
      + \int_{\R^n}\alpha \int_{0}^T \int_{0}^t \frac{(v(t)-v(s))^2}{(t-s)^{1+\alpha}} \ ds \ dt \ dx \\
     &\quad +\int_0^T \mathcal{B}(v,v) =0.
    \end{aligned}
   \]
  Then $v\equiv 0$.
 \end{proof}

\section{Appendix}
 Here in the appendix we provide the detail and computations that prove the existence of 
 a solution to the weak equation \eqref{e:w} via approximating solutions. For simplicity
 we write the operators in \eqref{e:w} and \eqref{e:firstd} as $\mathcal{H}$ and $\mathcal{H}_{\e}$
 respectively. 
 
 \begin{theorem}  \label{t:existence}
  For smooth bounded initial data , there exists a weak solution $w$ to \eqref{e:w}
  Furthermore, $w$ satisfies the estimates and conclusions of Lemma \ref{l:degiorgi1}, Corollary
  \ref{c:rewrite}, and Lemma \ref{l:second}.
 \end{theorem} 
 
 \begin{proof}
  Fix $\phi \in C_0^1((a,T)\times \R^n)$. There exists a sequence of solutions $w_{\epsilon}$ to 
  \eqref{e:firstd} with $\epsilon \to 0$. From the estimate \eqref{e:step1} we obtain the existence of
  $w$ such that 
   \[
     w_{\epsilon} \rightharpoonup w \in (H^{\alpha/2}(a,T) \times \R^n) \cup 
      ((a,T) \times H^{\sigma /2}(\R^n)). 
   \]
   and 
   \[
     w_{\epsilon} \to w \in L^p ((a,T)\times \R^n),
   \]
   for $p$ as defined in \eqref{e:p}. 
   We now label $\phi_{\e}(x,t)=\phi(x,\epsilon j)$ and $K_{\e}(x,t)=K(x,\epsilon j)$ if 
   $\epsilon (j-1) < t< \epsilon j$. $\mathcal{B}_{\e}$ is the bilinear
   form associated with $K_{\e}$. 
   We now show that for $\phi \in C_0^1((a,T)\times \R^n)$ 
   \begin{equation}  \label{e:sol}
    0=\mathcal{H}(w,\phi).
   \end{equation}
   We note that $\mathcal{H}(w,\phi)=\mathcal{H}(w,\phi)+ \mathcal{H}_{\epsilon}(w_{\epsilon},\phi)$, and
   we show that $\mathcal{H}(w,\phi)+ \mathcal{H}_{\epsilon}(w_{\epsilon},\phi) \to 0$. 
   We begin by showing that 
    \[
     \lim_{\e \to 0} \int_a^T \mathcal{B}_{\e}(w_{\e},\phi) \ dt 
     = \lim_{\e \to 0} \e \sum_{0<j\leq k} \mathcal{B}
      (w_{\e}(\e j, x), \phi(\e j,x))
    \]
   We first consider the region $|x-y| < M^{-1}$ for $M$ large. Now
    \[
     \begin{aligned}
      &\int_a^T \mathop{\int \int}_{|x-y|<M^{-1}} 
       K(x,y,t)(w_{\epsilon}(t,x)-w_{\epsilon}(t,y))(\phi(t,x)-\phi(t,y)) \\
         &\quad \leq \left(\int_a^T \mathcal{B}(w_{\epsilon},w_{\epsilon}) \right)^{1/2}
           \left(\int_a^T \mathop{\int \int}_{|x-y|< M^{-1}}K(x,y,t)(\phi(t,x)-\phi(t,y))^2 \right)^{1/2}.
     \end{aligned}
    \]
   Similarly,
    \[
     \begin{aligned}
      & \e \sum_{0<j\leq k} \mathop{\int \int}_{|x-y|<M^{-1}} 
       K(x,y,t)(w_{\e}(t,x)-w_{\e}(t,y))(\phi(t,x)-\phi(t,y)) \\
         &\leq \left(\int_a^T \mathcal{B}(w_{\e},w_{\e}) \right)^{1/2}
          \left( \mathop{\int \int}_{|x-y|<M^{-1}} 
       \Lambda|x-y|^{-n-\sigma}(\phi(\e j,x)-\phi(\e j,y))^2 \right)^{1/2}.
     \end{aligned}
    \]
   Both terms go to $0$ as $M \to \infty$ independent of $\e$.
   The second bound approaches the first as $\epsilon \to 0$. 
   For the moment we assume that $K$ is smooth for $|x-y|\geq M^{-1}$. Then since $K$ and $\phi$ are smooth 
    \[
     \begin{aligned}
     &\lim_{\e \to 0} \int_a^T \mathop{\int \int}_{|x-y|\geq M^{-1}}
          K(x,y,t)(w_{\e}(t,x)-w_{\e}(t,y))(\phi(t,x)-\phi(t,y)) \ dx \ dy \ dt \\
     &\quad = \lim_{\e \to 0} \int_a^T \mathop{\int \int}_{|x-y|\geq M^{-1}}
          K_{\e}(x,y,t)(w_{\e}(t,x)-w_{\e}(t,y))(\phi_{\e}(t,x)-\phi_{\e}(t,y)) \ dx \ dy \ dt 
     \end{aligned}
    \]  
   By approximating $K$ with kernels that are smooth for $|x-y|\geq M^{-1}$ we obtain 
    \[
     \begin{aligned}
     &\lim_{\e \to 0}\left|  \int_a^T \mathcal{B}(w_{\e},\phi) \ dt 
     - \epsilon \sum_{0<j\leq k} \mathcal{B}
      (w_{\e}(\e j, x), \phi(\e j,x)) \right| \\
      &\quad \leq C  \int_a^T  \mathop{\int \int}_{|x-y|< M^{-1}}K(x,y,t)(\phi(t,x)-\phi(t,y))^2.
     \end{aligned}
    \]
   We let $M \to \infty$. 
    
   We now focus on the pieces in time and first aim to show
    \[
     \begin{aligned}
     &\lim_{\e \to 0}\int_{\R^n} \int_0^T \int_0^t \frac{(w_{\e}(t)-w_{\e}(s))
       (\phi(t)-\phi(s))}{(t-s)^{1+\alpha}} \ ds \ dt \\
      &\quad = \lim_{\e \to 0}
      \mathop{\sum \sum}_{0\leq i<j\leq k} \int_{\e (j-1)}^{\e j} 
       \int_{\e (i-1)}^{\e i} \frac{(w_{\e}(\e j)-w_{\e}(\e i))
      (\phi(\e j)-\phi(\e i))}{(\e(j-i))^{1+\alpha}}.  \\
     \end{aligned}
    \]
   To do so we write $\phi(t)= \phi(t)-\phi_{e}(t)+\phi_{e}(t)$ and subtract the above two terms. 
   Since $\phi_{\e}(t) \to \phi(t)$ and $w_{\e} \rightharpoonup w$ in $H^{\alpha /2} \times \R^n$ we have that
    \[
      \lim_{\e \to 0} \left| \int_{\R^n} \int_0^T \int_0^t \frac{(w_{\e}(t)-w_{\e}(s))\left[(\phi(t)-\phi(s))
        -(\phi_{\e}(t)- \phi_{\e}(s))\right]}{(t-s)^{1+\alpha}} \ ds \ dt \right| \to 0.
    \]
   We now must show
    \begin{equation}  \label{e:1/3bound}
     \begin{aligned}
     0&= \lim_{\e \to 0}  \\
      &\mathop{\sum \sum}_{0\leq i<j\leq k} (w_{\e}(\e j)-w_{\e}(\e i))
      (\phi(\e j)-\phi(\e i)) 
       \int_{\e (j-1)}^{\e j} \int_{\e (i-1)}^{\e i} 
         \frac{1}{(t-s)^{1+\alpha}}-\frac{1}{(\e(j-i))^{1+\alpha}} 
     \end{aligned}
    \end{equation}
   We break up the integral over the sets $(t-s)\leq \e^{1/3}$ and $(t-s)> \e^{1/3}$.
   Another consequence of the strong and weak convergence of $\phi_{e}$ and $w_{e}$ is that 
    \[
     \begin{aligned}
     0 &= \lim_{\e \to 0} C \int_{\R^n} \mathop{\int \int}_{t-s\leq \e^{1/3}}
             \frac{|(w_{\e}(t)-w_{\e}(s))(\phi_{\e}(t)-\phi_{\e}(s))|}{(t-s)^{1+\alpha}} \\
       &\geq \lim_{\e \to 0} \int_{\R^n} \mathop{\sum \sum}_{\e (j-i) \leq \e^{1/3}}
         |(w_{\e}(\e j)-w_{\e}(\e i))(\phi(\e j)-\phi(\e i))| 
       \int_{\e (j-1)}^{\e j} \int_{\e (i-1)}^{\e i} 
       \frac{1}{(\e(j-i))^{1+\alpha}}
     \end{aligned}
    \]
   Now for $t-s>\epsilon^{1/3}$,
   and $\epsilon (i-1)\leq s \leq \epsilon i$ and $\epsilon (j-1)\leq t \leq \epsilon j$
    \[
     \begin{aligned}
     \left| (t-s)^{-(1+\alpha)} - (\epsilon(j-i))^{-(1+\alpha)}\right|
      &\leq (\epsilon^{1/3}-\epsilon)^{-(1+\alpha)} - (\epsilon^{1/3})^{-(1+\alpha)}\\
     &\leq (1+\alpha)(\epsilon^{1/3}-\epsilon)^{-(2+\alpha)}\epsilon \\
     &\leq (1+\alpha)(\epsilon^{1/3}/2)^{-(2+\alpha)}\epsilon \\
     &\leq C\epsilon^{(1-\alpha)/3}. 
     \end{aligned}
    \]
    and so \eqref{e:1/3bound} is holds. 
    
    The remaining pieces in time are handled similarly.

 \end{proof}

\bibliographystyle{amsplain}
\bibliography{refparabolic}

\providecommand{\bysame}{\leavevmode\hbox to3em{\hrulefill}\thinspace}
\providecommand{\MR}{\relax\ifhmode\unskip\space\fi MR }
\providecommand{\MRhref}[2]{%
  \href{http://www.ams.org/mathscinet-getitem?mr=#1}{#2}
}
\providecommand{\href}[2]{#2}
\begin{thebibliography}{1}

\bibitem{ccv11}
Luis Caffarelli, Chi~Hin Chan, and Alexis Vasseur, \emph{Regularity theory for
  parabolic nonlinear integral operators}, J. Amer. Math. Soc. \textbf{24}
  (2011), no.~3, 849--869. \MR{2784330 (2012c:45024)}

\bibitem{cc95}
Luis~A. Caffarelli and Xavier Cabr{\'e}, \emph{Fully nonlinear elliptic
  equations}, American Mathematical Society Colloquium Publications, vol.~43,
  American Mathematical Society, Providence, RI, 1995. \MR{1351007 (96h:35046)}

\bibitem{d04}
D.~del Castillo-Negrete, B.A. Carreras, and V.E. Lynch, \emph{Fractional
  diffusion in plasma turbulence}, Physics of Plasmas (2004).

\bibitem{d05}
\bysame, \emph{Nondiffusive transport in plasma turbulene: A fractional
  diffusion approach}, Physical Review Letters (2005).

\bibitem{mk00}
Ralf Metzler and Joseph Klafter, \emph{The random walk's guide to anomalous
  diffusion: a fractional dynamics approach}, Phys. Rep. \textbf{339} (2000),
  no.~1, 77. \MR{1809268 (2001k:82082)}

\bibitem{z13}
Rico Zacher, \emph{A {D}e {G}iorgi--{N}ash type theorem for time fractional
  diffusion equations}, Math. Ann. \textbf{356} (2013), no.~1, 99--146.
  \MR{3038123}

\bibitem{z02}
G.~M. Zaslavsky, \emph{Chaos, fractional kinetics, and anomalous transport},
  Phys. Rep. \textbf{371} (2002), no.~6, 461--580. \MR{1937584 (2003i:70030)}

\end{thebibliography}

\end{document}